\documentclass[12pt]{elsarticle}

\usepackage[utf8]{inputenc}

\usepackage{hyperref}
\hypersetup{
  pdfauthor={Daren A. Thimm, Zijia LI, Hans-Peter Schröcker, Johannes Siegele},
  pdftitle={Quadratic Motion Polynomials With Irregular Factorizations},
  pdfkeywords={conformal geometric algebra, conformal kinematics, rational motion, motion factorization, Villarceau motion, circular translation},
  hidelinks,
}

\usepackage{amsmath,amssymb,amsfonts,amsthm}
\usepackage{mathtools}
\usepackage{todonotes}
\usepackage{xspace}

\usepackage{subcaption}

\newtheorem{theorem}{Theorem}[section]

\newtheorem{proposition}[theorem]{Proposition}
\theoremstyle{definition}
\newtheorem{definition}[theorem]{Definition}
\newtheorem{example}[theorem]{Example}
\theoremstyle{remark}
\newtheorem{remark}[theorem]{Remark}

\begin{document}

\begin{frontmatter}
\journal{}

\title{Quadratic Motion Polynomials With Irregular Factorizations}

\author[1]{Daren A. Thimm}
\ead{daren.thimm@uibk.ac.at}
\address[1]{Universität Innsbruck, Department of Basic Sciences in Engineering Sciences, Technikerstr.~13, 6020 Innsbruck, Austria}

\author[2,3]{Zijia Li}
\ead{lizijia@amss.ac.cn}
\address[2]{State Key Laboratory of Mathematical Sciences, Academy of
  Mathematics and Systems Science, Chinese Academy of Sciences, Beijing 100190,
  China;}
\address[3]{School of Mathematical Sciences, University of Chinese Academy of Sciences, Beijing 100049, China}

\author[1]{Hans-Peter Schröcker}
\ead{hans-peter.schroecker@uibk.ac.at}

\author[4]{Johannes Siegele}
\ead{johannes.siegele@ricam.oeaw.ac.at}
\address[4]{Johann Radon Institute for Computational and Applied Mathematics (RICAM), Austrian Academy of Sciences, Altenberger Straße 69, 4040 Linz, Austria}

\begin{abstract}
  Motion polynomials are a specific type of polynomial over a Clifford algebra that can conveniently describe rational motions. There exists an algorithm for the factorization of motion polynomials that works in generic cases. It hinges on the invertibility of a certain coefficient occurring in the algorithm. If this coefficient is not invertible, factorizations may or may not exist. In the case of existence we call this an irregular factorization. We characterize quadratic motion polynomials with irregular factorizations in terms of algebraic equations and present examples whose number of unique factorizations range from one to infinitely many. For two special sub-cases we show the unique existence of such polynomials. In case of commuting factors we obtain the conformal Villarceau motion, in case of rigid body motions the circular translation.
 \end{abstract}


\end{frontmatter}

\section{Introduction}
\label{sec:introduction}

The factorization theory of motion polynomials was introduced in
\cite{hegedus13} with the purpose of constructing closed-loop linkages directly
from the motion of one link. Ever since, it saw numerous applications in
mechanism science, cf. \cite{hegedus15,liu23,liu24,liu21} to name but a few. But
the factorization theory is also interesting in its own right. It extends
classical results on the factorization of unilateral quaternionic polynomials
\cite{gordon65,niven41} to dual quaternionic polynomials that parametrize
rational motions.

While generically a motion polynomial of degree \(n\) admits \(n!\)
factorizations with linear factors over both the quaternions \(\mathbb{H}\) and the
dual quaternions \(\mathbb{DH}\), a notable difference between these two algebras is
that motion polynomials over \(\mathbb{DH}\) might also have infinitely many or no
factorization. While the case of zero factorizations has been resolved recently
\cite{li25}, not much is known about motion polynomials with infinite
factorizations. The most basic example is a quadratic motion polynomial that
parametrizes the curvilinear translation along a circle (a \emph{circular
  translation} for short). We will re-visit it in
Section~\ref{sec:rigid-body-motions}. The kinematic explanation for its unusual
factorization properties is the possibility to generate this motion in
infinitely many ways by a parallelogram linkage, as seen in Figure~\ref{fig:parallelogram}.
\begin{figure}
    \centering
    \includegraphics[width=0.8\textwidth]{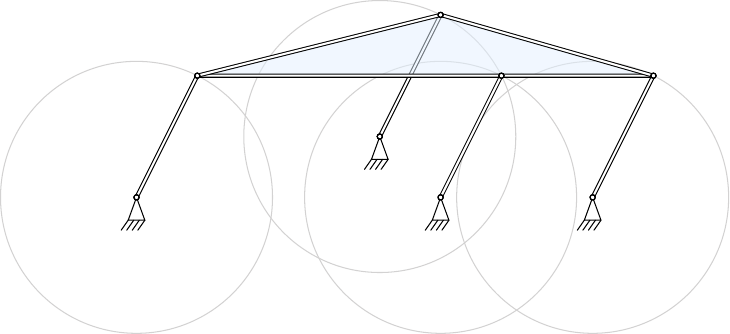}
    \caption{Construction of a circular translation, by parallelogram linkages.}
    \label{fig:parallelogram}
\end{figure}

The concept of motion polynomials\footnote{We finally settled with this name
  after calling them ``rotor polynomials'' in the proposal acknowledged at the
  end of this text and ``spinor polynomials'' in \cite{kalkan22}. Both earlier
  names led to confusion in relevant scientific communities.} and the
factorization algorithm of \cite{hegedus13} readily extends from dual
quaternions \(\mathbb{DH}\) to conformal geometric algebra \(\ensuremath{\mathrm{CGA}}\) or from Euclidean
kinematics to conformal kinematics \cite{li18}. The most notable difference is
that a generic motion polynomial of degree \(n\) generically admits up to
\((2n)!/2^n\) factorizations. Even in this larger algebra, only one further  non-trivial motion polynomial with infinite factorizations is known (cf. Definition~\ref{def:trivial}). Motivated
by some applications to physics, it was introduced by L.~Dorst in
\cite{dorst19}. Its infinitely many factorizations were described in
\cite{icgg_paper} and it is noteworthy that they all commute. Since the motion's trajectories are related to Villarceau circles on a
torus, as can be observed in Figure~\ref{fig:villarceau}, we call it the \emph{conformal Villarceau motion,} cf.
Section~\ref{sec:commuting}.

\begin{figure}
    \centering
    \includegraphics[width=0.8\textwidth]{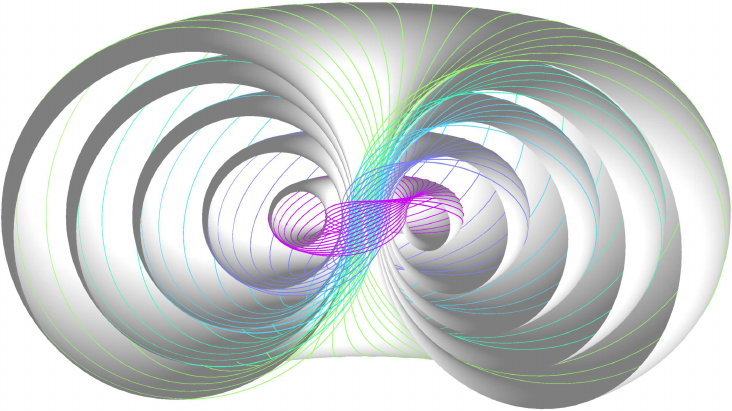}
    \caption{Trajectories of the Villarceau motion.}
    \label{fig:villarceau}
\end{figure}

The existence of infinitely factorizable motion polynomials can be traced back
to the non-invertibility of a specific coefficient which arises during the
factorization algorithm. Conversely, the non-invertibility of this coefficient
does not immediately imply the existence of an infinite amount of
factorizations. It is possible that no factorizations exist but also a finite
number of factorizations is still possible --- a phenomenon that has not yet
been observed in literature. We call factorizations obtained under these
conditions \emph{irregular.} A precise definition will be given in Definition~\ref{def:irregular}.

In this paper we study quadratic motion polynomials over \(\ensuremath{\mathrm{CGA}}\) with irregular
factorizations. In Section~\ref{sec:motion-polynomials} we characterize them as
real solutions to a system of algebraic equations. Each linear factor
parametrizes one of three possible \emph{simple motions} in the sense of
\cite{dorst16} --- a conformal rotation, a transversion, or a conformal scaling.
We show by example that all possible pairings of these motion types can
appear as irregularly factorizable motions.

The primary approach involves examining the conditions under which an algebra
element becomes non-invertible. The characteristics of non-invertible vectors in
Conformal Geometric Algebra (\ensuremath{\mathrm{CGA}}) have been extensively investigated in the
literature \cite{dorst16,dorst2009geometric}. These findings have significant
applications in the field of Automated Theorem Proving, as demonstrated in
several studies \cite{lihongbo21,lihongbo19,lihongbo11}. This paper extends the
investigation to non-invertible elements that go beyond mere null vectors.

In Section~\ref{sec:rigid-body-motions} we use our characterization of irregular
factorizability to show that the circular translation is the only irregularly factorizable rigid body motion.
Finally, in Section~\ref{sec:commuting} we prove that \emph{commuting} irregular factors imply that
the motion is the already known conformal Villarceau motion.
Both uniqueness statements
ignore some trivial exceptions and are only up to conformal equivalence.

\section{Preliminaries}

In this section we give an overview of all of the necessary concepts used in the rest of this article.

We will explain how to model conformal transformations, how to apply them continuously to geometric objects and how to
decompose more complex motions into their simple constituent parts.

\subsection{Conformal Geometric Algebra}

To describe conformal motions, which contain the Euclidean motions as a subset, we will use the framework of conformal geometric algebra also known as \ensuremath{\mathrm{CGA}} \cite{bayro-corrochano19}. It is a Clifford algebra of signature \((4,1)\) over the reals together with an involution \(a \mapsto \tilde{a}\) called \emph{reversion.} We choose the orthonormal basis \(\{e_1,e_2,e_3,e_+,e_-\}\in \mathbb{R}^{4,1}\) such that
\begin{equation*}
    e_1^2=e_2^2=e_3^2=e_+^2=1, \quad e_-^2=-1.
\end{equation*}
For \(i \neq j \in \{1,2,3,+,-\}\) we define \(e_i e_j = -e_j e_i \coloneqq e_{ij}.\)
This can be extended in the same way to the product of multiple basis vectors.
The elements of \ensuremath{\mathrm{CGA}} consist of linear combinations of all possible multiplications of the basis vectors.
We say an element has grade \(n\) if it can be written as a linear combination of products of \(n\) basis vectors.
In this case we call it an \(n\)-vector. If an element does not have a unique grade, but is rather the sum of multiple elements with a single grade, we call it a multivector.

The involution  \(a \mapsto \tilde{a}\) is defined by reversing the order of the indices in each multiplication.
By the anti-commutativity of the basis vectors, this is equivalent to a sign change according to the parity of the number of transpositions needed to invert the list of indices.

For our approach, a projective viewpoint is more natural and we will often consider multivectors as points of the projective space \(\mathbb{P}(\ensuremath{\mathrm{CGA}})\) over~\ensuremath{\mathrm{CGA}}. If \(a \in \ensuremath{\mathrm{CGA}}\) then we denote the corresponding point in \(\mathbb{P}(\ensuremath{\mathrm{CGA}})\) by \([a]\).

A conformal displacement is a successive inversion in a number of spheres. The sphere $s$ with center \((c_x,c_y,c_z)\) and radius \(r\) is embedded in \ensuremath{\mathrm{CGA}} as
\begin{equation*}
  s = c_x e_1+c_y e_2+c_z e_3+\frac{q-1}{2} e_+ +\frac{q+1}{2}e_-,
\end{equation*}
where \(q \coloneqq c_x ^2 + c_y^2 + c_z^2 - r^2\).

Points and planes are viewed as special cases of spheres.
For points we let the radius be zero and planes can be thought of taking the limit of the center and the radius going to inifinity.
From this it follows that the point $p$ at \((p_x,p_y,p_z)\) and the plane $\pi$ with normal vector \((n_x,n_y,n_z)\) and distance \(d\) to the origin are represented as
\begin{equation*}
  p = p_x e_1+p_y e_2+p_z e_3+\frac{p_x ^2 + p_y^2 + p_z^2-1}{2} e_+ +\frac{p_x ^2 + p_y^2 + p_z^2+1}{2}e_-,
\end{equation*}
and
\begin{equation*}
  \pi = n_x e_1+n_y e_2+n_z e_3+(n_x^2+n_y^2+n_z^2)d(e_+ + e_-).
\end{equation*}

The inversion of an element \([a]\) by a sphere \(s\) is given by \([s a \tilde{s}]\), commonly known as the sandwich product. As already mentioned, using equivalence classes is beneficial in the context of motion polynomials. In this viewpoint we loose the weight and orientation of objects but gain ease of use for polynomial computations. Instead of the sphere \(s\) we could equally well use \([s]\). Hence we also do not require the normalization \(s\tilde{s} = \pm 1\).

Later we will study conformal motions, that is, continuous sets of displacements parameterized by rational functions (or polynomials in the projective setting). Since composition with a fixed sphere inversion is irrelevant in this context, we restrict to the even sub-algebra \ensuremath{\mathrm{CGA}_+} of \ensuremath{\mathrm{CGA}}. It corresponds to the composition of an \emph{even} number of sphere inversions.

An element \([a] \in \mathbb{P}(\ensuremath{\mathrm{CGA}_+})\) describes a conformal displacement if and only if \(a\tilde a\), \(\tilde a a \in\mathbb{R} \setminus \{0\}\). Elements \([a] \in \mathbb{P}(\ensuremath{\mathrm{CGA}_+})\) fulfilling the condition \(a\tilde a = \tilde a a \in\mathbb{R}\) lie on an algebraic variety defined by it, called the \emph{Study variety} \(\ensuremath{\mathcal{S}}\) of conformal kinematics \cite{kalkan22}.

Following a suggestion by N.~Wildberger \cite{DivineProportions}, we call \(a\tilde a\) the quadrance of \(a\) if \(a\tilde a \in \mathbb{R}\). Sometimes this is also called the norm of \(a\), but since it is rather a squared norm we chose the former name to avoid confusion. Note that we do not have to distinguish a left and right quadrance \(a\tilde a\), \(\tilde a a\) since both values coincide.

To describe smooth motions we can use the typical approach of taking the rotor exponential.
Let \(B\in\ensuremath{\mathrm{CGA}}\) be a 2-blade. The motion described by this blade and parametrized by a
time \(u\) is then given by sandwiching with the exponential \(e^{uB/2}\).
Using the fact that
 \begin{equation*}
  e^{uB^\prime} =
     \left\lbrace
     \begin{array}{l}
       \cos (u \sqrt{B^\prime \tilde B^\prime}) + \frac{B^\prime}{\sqrt{B^\prime\tilde B^\prime}} \sin (u\sqrt{B^\prime \tilde B^\prime})\\
       1 + B^\prime u\\
       \cosh (u\sqrt{-B^\prime \tilde B^\prime}) + \frac{B^\prime}{\sqrt{-B^\prime\tilde B^\prime}}\sinh (u\sqrt{-B^\prime \tilde B^\prime})
     \end{array}\right.
     \begin{array}{c}
       \text{ if } B^\prime\tilde B^\prime > 0\\
       \text{ if } B^\prime\tilde B^\prime = 0\\
       \text{ if } B^\prime\tilde B^\prime < 0
     \end{array},
\end{equation*} we can rewrite the exponential \(e^{uB/2}\), as a linear polynomial function \(t - a\)
via an appropriate choice of the following reparametrizations.
\[t\coloneqq\cot(u), \quad t\coloneqq u^{-1}, \quad t\coloneqq \pm \coth(u)\]
For \(a\in \ensuremath{\mathrm{CGA}_+}\) we call \(t-a\) a simple motion if \((t-a)(t-a){\,\tilde{}} \in \mathbb{R}\) for any \(t \in \mathbb{R}\).
These can be categorized into three groups according to the number of distinct real roots of the quadrance-polynomial \((t-a)(t-a){\,\tilde{}}\).
For zero roots it describes a rotation, for one distinct root a translation and for two a scaling 
\cite{dorst16,kalkan22}.

The number of roots correspond to the number of distinct points of a
transformation to which all points in space are mapped simultaneously, as seen in Figure \ref{fig:elementary-motions}.
For a rotation there exists no such point. For a translation this point is the point at infinity, or a conformal image thereof. For a scaling the points are the center of scaling and the point at infinity, or their conformal images.

\begin{figure}
\centering
\includegraphics[width=0.30\textwidth]{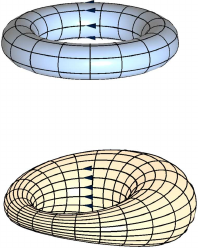}
\includegraphics[width=0.30\textwidth]{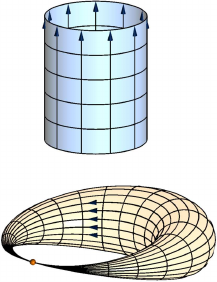}
\includegraphics[width=0.35\textwidth]{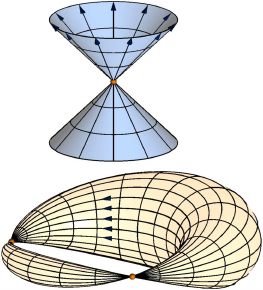}
\caption{Simple conformal motions and euclidean counterpart. From left to right: rotation, transversion, scaling}
\label{fig:elementary-motions}    
\end{figure}

Letting multiple simple motions act upon an element can be done by letting them act individually in sequence. This corresponds to the multiplication of the individual linear polynomials that represent the smooth motion.

\subsection{Polynomials in \ensuremath{\mathrm{CGA}_+}}\label{sub:Polynomials_in_CGA}
We will now take a closer look at polynomials in \ensuremath{\mathrm{CGA}_+}.  Let \(C = \sum_{i=0}^{n} c_i t^i\), where \(c_i \in \ensuremath{\mathrm{CGA}_+}\) be a polynomial in the indeterminate \(t\). We define the multiplication with the convention that the indeterminate commutes with the coefficients of the polynomials. This is reasonable given the fact the we regard \(t\) as a real motion parameter (time). Nonetheless, we will also evaluate polynomials at more general algebra elements, but then we need to distinguish two types of evaluation, the right and left evaluation. They are defined by
\[
  C(h)_r \coloneqq \sum_{i=0}^{n}c_ih^i
  \quad\text{and}\quad
  C(h)_l \coloneqq \sum_{i=0}^{n}h^ic_i,
\]
respectively. From now on we will only be using the right evaluation and the emerging theory and use the shorter notation \(C(h) \coloneqq C(h)_r\). Everything can be also formulated equivalently for the ``left'' theory.

We define the reversion \(\tilde{C}\) of \(C\) by taking the reversion of all coefficients. Furthermore, we define the action of \(C\) on an element \(a \in \ensuremath{\mathrm{CGA}}\) by the sandwich product \(C a \tilde{C}\). For \(C\) to describe a conformal motion it is necessary for its left and right quadrance to be equal and a non-zero real polynomial: \(C\tilde{C} = \tilde{C}C \in \mathbb{R}[t] \setminus \{0\}\). In the case that these conditions are met we call \(C\) a \emph{motion polynomial.} Indeed, it describes a conformal motion as \(C(t)\) is a conformal displacement for any \(t \in \mathbb{R}\) with at most finitely many exceptions, namely the real roots of the quadrance polynomial \(C\tilde{C}\). Since the trajectories of all points are rational curves we speak of a \emph{rational motion.}

\subsection{Decomposition into Simple Motions}\label{sub:Decomposition_into_Simple_Motions}

The question now arises how to decompose a rational motion into simple motions, that is, how to factor a motion polynomial \(C\) into linear factors. Over a non-commutative ring this is a non-trivial matter. Some things are already known about the factorizability of such polynomials \cite{li18}.
\begin{itemize}
  \item For generic polynomials \(C\) of degree \(n\) the number of factorizations depends upon the number of real roots of \(C\tilde C\)
    and ranges from \(n!\) for no real roots to \(\frac{(2n)!}{2^n}\) for \(2n\) roots.
  \item There exist polynomials with no factorization.
  \item There exist polynomials with infinitely many factorizations.
  \item \(t-h\) is a right factor of \(C\) if and only if \(h\) is a right root of \(C\)
\end{itemize}

We will now describe a method to compute all factorizations with linear factors
of a motion polynomial, provided they exist. Our exposition follows \cite{li18}
but it should be mentioned that more or less similar factorization algorithms
have been described at many different places and in different context
\cite{gordon65,serodio01,janovska10,kalantari13,falcao17,sakkalis19,gentili21,izosimov23,koprowski23,li24,hoffmann24}. It should also be noted that, due to the restricted nature of
polynomials we consider, our topic is a special case of the more general and
well-developed factorization theory over rings, cf. the survey
paper~\cite{smertnig16}.

For this we will henceforth assume \(C\) to have an invertible leading
coefficient. Because of \(C\tilde{C} \neq 0\) this is no loss of generality as it
can be ensured by a suitable rational re-parametrization. But then we can also
assume that \(C\) is monic, i.e. has leading coefficient \(1\),
as the leading coefficient can be factored out.

Since \(t-h\) being a right factor is equivalent to \(h\) being a right root of
\(C\), the question of finding right factors reduces to the question of finding
right roots. Since a right root of the polynomial is also necessarily a right
root of the quadrance of the polynomial, we can search for roots of \(C\tilde{C}\).
Each right factor \(t-h\) and therefore also each right root corresponds to a
monic, quadratic factor \(M\coloneqq(t-h)(t-h){\,\tilde{}}\) of the quadrance
polynomial. Using the Euclidean algorithm \cite{li18}, we can now
divide \(C\) by \(M\) and get \[C = QM + R\] for suitable polynomials
\(Q\), \(R\) with \(\deg R < \deg M = 2\).
This is possible because the leading coefficient \(1\) of the divisor \(t-h\) is invertible.
Since we know that \(h\) is a root of \(C\) and \(M\) it follows that it
also has to be a root of \(R\) \cite[Lemma~1]{hegedus13}. Furthermore, we can
write \(R = r_1t + r_0\), where \(r_0,r_1\in \ensuremath{\mathrm{CGA}_+}\). The question of
factorization is closely related to the
roots of the linear polynomial \(R\). Assuming for the moment that \(r_1\) is
invertible we get the unique solution \(h=-r_1^{-1}r_0\) for our particular
choice of \(M\). Once we have found this right factor we can divide \(C\) by it
and get \(C^\prime\) of a lower degree and can iterate. Note that
there is also a ``left'' variant of this factorization algorithm.

If the leading coefficient \(r_1\) is not invertible, we potentially get no or
infinite roots of \(R\). To solve for \(h\) we convert \(R(h)=0\) into a system
of linear equations via the coefficient-vectors with regards to the basis of
\ensuremath{\mathrm{CGA}}. The resulting system of linear equations has no or an infinite amount of
solutions which may lead to no, or infinitely many but also to finitely many
factors of \(C\) that can be determined by imposing further necessary condition
(cf.~Section~\ref{sec:irregular-factorizations}). In order to capture this
special situation, we define

\begin{definition}
  \label{def:irregular}
  We call a factorization \(C = (t-h_1)(t-h_2) \cdots (t-h_n)\) of a motion
  polynomial \(C\) \emph{irregular,} if there exists an index \(\ell \in
  \{1,2,\ldots,n\}\) such that the linear remainder polynomial \(R\) obtained by
  dividing either
  \begin{itemize}
  \item \((t-h_1)(t-h_2) \cdots (t-h_\ell)\) or
  \item \((t-h_\ell)(t-h_{\ell+1}) \cdots (t-h_n)\)
  \end{itemize}
  by \(M \coloneqq (t-h_\ell)(t-\tilde{h}_\ell)\) has a non-invertible leading
  coefficient. In this case we also say that \(C\) is \emph{irregularly
    factorizable.}
\end{definition}

\begin{remark}
  Definition~\ref{def:irregular} covers both the left and the right
  version of the factorization algorithm. It is, however, not clear whether an
  irregular ``right'' factorization implies an irregular ``left'' factorization
  or not. For quadratic motion polynomials however, 
  as the remainder polynomials have the same leading coefficient, both notions coincide.
\end{remark}

\section{Motion Polynomials with Irregular Factorizations}
\label{sec:motion-polynomials}

Now that we have a framework with which to describe rational motions in \ensuremath{\mathrm{CGA}} and
we know how to factorize them, we want to investigate the special class of
irregularly factorizable motions. As stated in
Section~\ref{sub:Decomposition_into_Simple_Motions}, in general a motion polynomial only has a
finite amount of factorizations and they can be computed by a straightforward
algorithm which, however, may fail in some instances. It is precisely those
irregular cases that we are interested in.

As of today, only few examples of irregularly factorizable motion polynomials
are known, the most famous of which are the circular translation \cite{li19} that is, a motion without rotational component whose trajectories are circles 
and the Villarceau motion \cite{dorst19} whose trajectories are related to Villarceau circles on a
torus. Our aim is to find a complete description of motion polynomials of
degree two with irregular factorizations and provide new examples. For this we
need to answer the question how irregular factorizability can be stated
algebraically. We will only be regarding monic polynomials as, by assumption,
the leading coefficient is invertible and can therefore be factored out.

\subsection{Conditions for Irregular Factorizability}
\label{sec:conditions_for_irregular_factorizability}

Let us regard a polynomial \(C \coloneqq t^2 + a t + b \in\ensuremath{\mathrm{CGA}_+}[t]\) and assume
that \(C\) has a factorization as \(C = (t-h_1)(t-h_2)=t^2-(h_1+h_2)t +h_1h_2\)
for some \(h_1,h_2\in\ensuremath{\mathrm{CGA}_+}\). When does \(C\) now have irregular factorizations?

\begin{theorem}
  \label{th:irregular}
  The factorization \(C = (t-h_1)(t-h_2)\) of the monic quadratic motion
  polynomial \(C \in \ensuremath{\mathrm{CGA}_+}[t]\) is irregular if and only if \(h_1-\tilde{h}_2\) is
  not invertible.
\end{theorem}

\begin{proof}
  To prove the theorem we try to factorize the polynomial. Since its
  norm-polynomial is given by \(C\tilde C =
  (t-h_1)(t-h_1){\,\tilde{}}\;(t-h_2)(t-h_2){\,\tilde{}}\) we can take
  \(M\coloneqq(t-h_2)(t-h_2){\,\tilde{}}\) as a monic quadratic factor. If we now
  divide \(C\) by \(M\) we see
  \begin{align*}
    C = M - (h_1-\tilde h_2)t + (h_1-\tilde h_2)h_2.
  \end{align*}
  From this we can follow that the factorization is irregular if and only if
  \(h_1-\tilde h_2\) is not invertible, as explained in
  Subsection~\ref{sub:Polynomials_in_CGA}.

  Taking \(M\coloneqq(t-h_1)(t-h_1){\,\tilde{}}\) will result in the same criterion.
\end{proof}

\begin{remark}
  Note that our formulation of Theorem~\ref{th:irregular} assumes existence of a
  factorization. The polynomial \(C\) allows for at least one and possibly
  infinitely many factorization.
\end{remark}

\subsection{Non-Invertibility Condition}

To complement Theorem~\ref{th:irregular} we will now derive an algebraic
formulation of when an algebra element is non-invertible. While this is known in
the Geometric Algebra community, we strive for a simplified algebraic criterion to make ensuing
calculations more manageable.

\begin{proposition}
  \label{prop:invertible}
  Let \(a\in\ensuremath{\mathrm{CGA}}\). Then \(a\) is invertible if and only if \(a\tilde{a}\) is
  invertible.
\end{proposition}

\begin{proof}
  If \(a\) is invertible, then \(\tilde{a}\) is invertible, \((\tilde{a})^{-1} =
  (a^{-1}){\,\tilde{}}\) and the inverse of \(a\tilde{a}\) is \((\tilde{a})^{-1}a^{-1}\).
  Conversely, if \(a\tilde{a}\) is invertible, then the inverse of \(a\) is given
  by \(\tilde{a} (a\tilde{a})^{-1}\).
\end{proof}

Proposition~\ref{prop:invertible} allows us to reduce the dimensions in which we
have to look for an inverse from 32 to twelve: Since \(a\tilde{a}\) is its own
reverse it consists only of grade \(0\), \(1\), \(4\) and \(5\) multivectors,
which are all blades and are not changed by reversion.

For this reduced set of \ensuremath{\mathrm{CGA}} we now explicitly calculate the determinant. On
this restricted set the determinant has a rather handy form. To calculate it, we
use the embedding of \ensuremath{\mathrm{CGA}} into \(\operatorname{Mat}_4(\mathbb{C})\), given by
the following mapping of the generators of the algebra \cite{MatrixRep}. All other
basis elements can be constructed via multiplication:
\begin{gather*}
  e_1 \mapsto \begin{pmatrix} 0 & -\mathrm{i} & 0 & 0\\ \mathrm{i} & 0 & 0 & 0\\ 0 & 0 & 0 & -\mathrm{i}\\ 0 & 0 & \mathrm{i} & 0\end{pmatrix},\quad
  e_2 \mapsto \begin{pmatrix} -1 & 0 & 0 & 0\\ 0 & 1 & 0 & 0\\ 0 & 0 & -1 & 0\\ 0 & 0 & 0 & 1\end{pmatrix},\quad
  e_3 \mapsto \begin{pmatrix} 0 & 0 & 0 & 1\\ 0 & 0 & 1 & 0\\ 0 & 1 & 0 & 0\\ 1 & 0 & 0 & 0\end{pmatrix},\\
  e_+ \mapsto \begin{pmatrix} 0 & 1 & 0 & 0\\ 1 & 0 & 0 & 0\\ 0 & 0 & 0 & -1\\ 0 & 0 & -1 & 0 \end{pmatrix},\quad
  e_- \mapsto \begin{pmatrix} 0 & 0 & 0 & 1\\ 0 & 0 & 1 & 0\\ 0 &-1 & 0 & 0\\ -1& 0 & 0 & 0 \end{pmatrix}.
\end{gather*}

Using this embedding, we compute the determinant of 
\begin{multline*}
  x = x_{0} + x_{1} e_{1} + x_{2} e_{2} + x_{3} e_{3} + x_{+} e_{+} + x_{-} e_{-}\\
  +  x_{123+} e_{123+} + x_{123-} e_{123-} + x_{12+-} e_{12+-} + x_{13+-}
  e_{13+-} + x_{23+-} e_{23+-}\\
  + x_{123+-} e_{123+-}\end{multline*}
consisting only of grades \(0\), \(1\), \(4\) and \(5\) as
\begin{equation}
  \label{eq:1}
  \det x = (q-2m)(q+2m) - 4\mathrm{i} qm
  = (q - 2\mathrm{i} m)^2
\end{equation}
where
\begin{align*}
  q &=
   x_{0}^2 -x_{1}^2 -x_{2}^2 -x_{3}^2 -x_{+}^2 +x_{-}^2 \\
  & \quad -x_{123+}^2 +x_{123-}^2 +x_{12+-}^2 +x_{13+-}^2 +x_{23+-}^2
  -x_{123+-}^2\\
  m &= x_{0} x_{123+-} - x_{1} x_{23+-} + x_{2} x_{13+-} - x_{3} x_{12+-} + x_{+} x_{123-} - x_{-} x_{123+}.
\end{align*}
It is zero if and only if \(m=q=0\) i.e. if both factors of the complex part are
zero. Equation~\eqref{eq:1} will be our preferred way to encode
non-invertibility in computations. More precisely, by
Proposition~\ref{prop:invertible}, \(a \in \ensuremath{\mathrm{CGA}}\) is not invertible if and only
if \(x \coloneqq a\tilde{a}\) satisfies \(\det(x)=0\).

\subsection{Irregular Factorizations}
\label{sec:irregular-factorizations}

In order to actually compute examples of irregularly factorizable quadratic
motion polynomials we can now proceed as follows:
\begin{enumerate}
\item We prescribe a linear right motion polynomial factor \(t - h_2\) and define
  \(M \coloneqq (t-h_2)(t-h_2){\,\tilde{}}\).
\item We compute \(h_1\) subject to the conditions that
  \begin{equation}
    \label{eq:2}
    \det(h_1-\tilde{h}_2) = 0
  \end{equation}
  and \(t-h_1\) is a motion polynomial. According to \cite{kalkan22}, the latter
  is the case if and only if
  \begin{equation}
    \label{eq:3}
    h_1\tilde{h}_1 \in \mathbb{R},\quad h_1 + \tilde{h}_1 \in \mathbb{R}.
  \end{equation}
\item The polynomial \(C \coloneqq (t - h_1)(t - h_2)\) is then irregularly
  factorizable.
\end{enumerate}

\begin{remark}
  \label{rem:two-things}
  A few things should be mentioned here:
  \begin{enumerate}
  \item While it is guaranteed that \(C\) is irregularly factorizable it is not
    clear whether it has a finite or an infinite amount of factorizations.
  \item By Proposition~\ref{prop:invertible} we can encode the vanishing of the
    determinant \eqref{eq:2} as vanishing of the determinant of the
    \emph{quadrance} of $h_1-\tilde{h}_2$. By Equation~\eqref{eq:1} this imposes
    two \emph{real} constraints on the unknown coefficients of~$h_1$.
  \item There is some evidence (cf. Remark~\ref{rem:rare}) that real solutions
    to the system of algebraic equations \eqref{eq:2}, and
    \eqref{eq:3} are rare. The real dimension of the solution variety seems to
    be smaller than its complex dimension.
  \end{enumerate}
\end{remark}

There are three conformally non-equivalent types of motions described by
\(t-h_2\) \cite{dorst16,kalkan22}, conformal rotation, transversion, and
conformal scaling. They are distinguished by the number of real roots of the
quadrance polynomial $(t-h_2)(t-\tilde{h}_2)$. In general, we expect there to exist
a quadratic motion polynomial $C = (t-h_1)(t-h_2)$ with irregular factorization
for each pairing of these three motion types. When demonstrating this by
example, we do not have to take care of the order of factors as taking the
reversion of the polynomial preserves the type of motions involved but switches
the motion type of the first and second factor. 

We will now investigate these different pairings. Searching for specific types of
motions of the first factor can be done by additionally prescribing that the
norm polynomial of this factor has to have zero, one or two real roots. Using
this we can now solve for the irregular factorizability condition with this added
restriction and find examples for each type of motion pairing,
as can be seen in Figure~\ref{fig:examples}.

\begin{remark}
  \label{rem:rare}
  All solutions of the irregular factorizability condition
  lie on the real variety generated by the determinant of
  the quadrance of an element. Interestingly, all real solutions seem so be singular
  points of this variety. This has been verified using the Mathematica
  \textsc{Resolve} command and checking if there exists a real non-singular
  point on the variety.
\end{remark}
\begin{example}
There exists at least one example for each case:
 \begin{description}
   \item[Rotation with Rotation]
   Seen in Figure~\ref{fig:r_r}.
   
   $h_1 = -e_{12} + e_{13} + e_{1-} + e_{23} + e_{2-} $,
   
   $h_2 = e_{12}$
   
   This motion has an infinite amount of factorizations.
   \item[Transversion with Rotation]
   Seen in Figure~\ref{fig:t_r}.
   
   $h_1 = -\frac{1}{2}e_{12} + \frac{4}{5}e_{13} + \frac{49}{30}e_{1-} + \frac{4}{3}e_{1+}$,
   
   $h_2 = e_{12}$
   
   This motion has one unique factorization.
   \item[Scaling with Rotation]
   Seen in Figure~\ref{fig:s_r}.
   
   $h_1 = e_{13} + \frac{\sqrt{6}}{3}(2 e_{1-} +1)$,
   
   $h_2 = e_{12}$ 
   
   This motion has one unique factorization.
   \item[Transversion with Transversion]
   Seen in Figure~\ref{fig:t_t}.
   
   $h_1 = e_{3-}-e_{+-}-e_{13}+e_{1+}+e_{1-}-e_{23}+e_{2+}+e_{2-}+\sqrt{2} $,
   
   $h_2 = e_{3+}+e_{3-}$
   
   This motion has two distinct factorizations.
   \item[Scaling with Scaling]
   Seen in Figure~\ref{fig:s_s}.
   
   $h_1 = -e_{3-}+e_{2-}+\sqrt{3}$,
   
   $h_2 = -e_{+-} $ 
   
   This motion has five distinct factorizations.
   \item[Transversion with Scaling]
   Seen in Figure~\ref{fig:t_s}.
   
   $h_1 = e_{2-} + \frac{1}{2}(\sqrt{5}e_{2+} + e_{+-})  $,
   
   $h_2 = -e_{+-}$ 
   
   This motion has three distinct factorizations.
 \end{description}
\end{example}

\begin{figure}
  \begin{subfigure}{0.4\textwidth}
    \includegraphics[width=\textwidth]{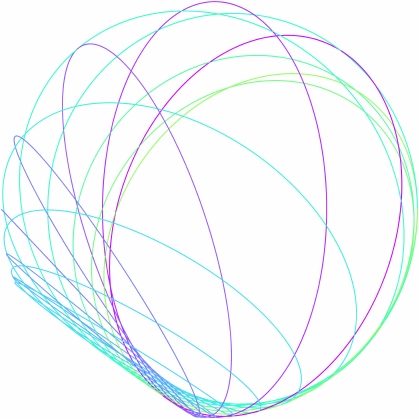}
    \caption{Rotation with Rotation}
    \label{fig:r_r}
  \end{subfigure}
  \hfill
  \begin{subfigure}{0.4\textwidth}
    \includegraphics[width=\textwidth]{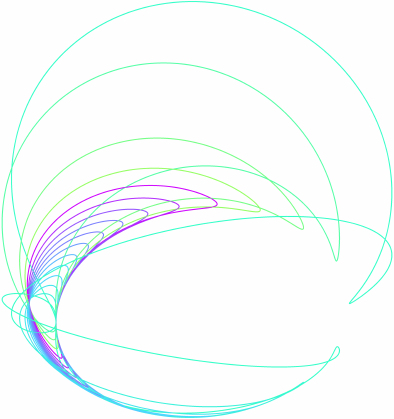}
    \caption{Transversion with Rotation}
    \label{fig:t_r}
  \end{subfigure}
  \\
  \begin{subfigure}{0.4\textwidth}
    \includegraphics[width=\textwidth]{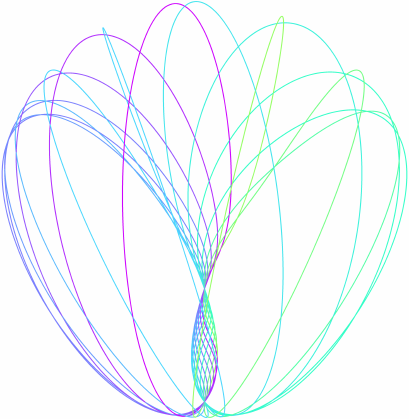}
    \caption{Scaling with Rotation}
    \label{fig:s_r}
  \end{subfigure}
  \hfill
  \begin{subfigure}{0.4\textwidth}
    \includegraphics[width=\textwidth]{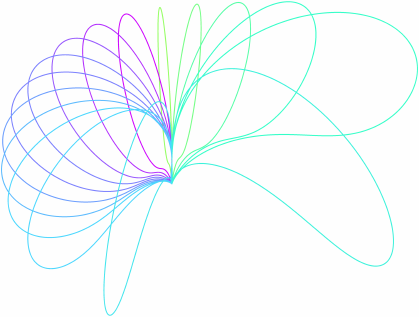}
    \caption{Transversion with Transversion}
    \label{fig:t_t}
  \end{subfigure}
\\
  \begin{subfigure}{0.4\textwidth}
    \includegraphics[width=\textwidth]{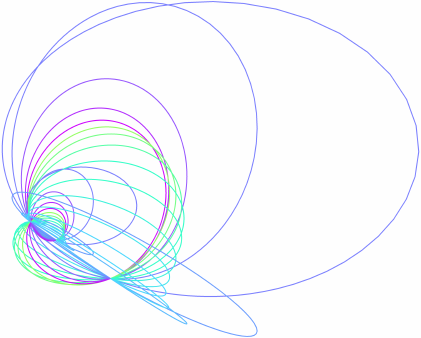}
    \caption{Scaling with Scaling}
    \label{fig:s_s}
  \end{subfigure}
  \hfill
  \begin{subfigure}{0.4\textwidth}
    \includegraphics[width=\textwidth]{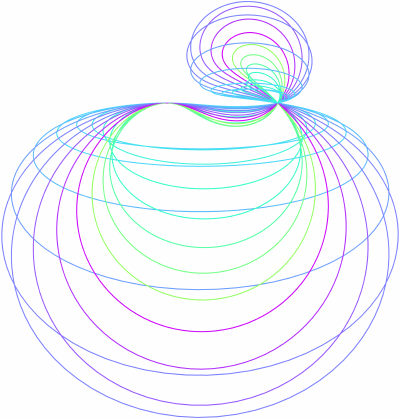}
    \caption{Transversion with Scaling}
    \label{fig:t_s}
  \end{subfigure}
  \caption{Trajectories of points of a circle under irregularly factorizable motion polynomials.}
  \label{fig:examples}
\end{figure}

\begin{remark}
  The graphics in Figure \ref{fig:examples} were plotted using \textsc{Asymptote}.
  An interactive version of the figures can be reproduced using the python package \textsc{cga-py} \cite{Thimm_2025} and plotting the trajectories for a unit circle in the \(x\)-\(y\)-Plane centered at \((0,0,3)\).
\end{remark}

\section{Rigid Body Motions}
\label{sec:rigid-body-motions}

A natural question arising now is, do there exist irregularly factorizable
motions in the space \(\operatorname{SE}(3)\) of rigid body motions? The only known
non-trivial example of degree two is the circular translation given by \(C
\coloneqq (t-(x e_{1}+y e_{2})(e_-+e_+))(t-e_{12})\) \cite{li15}. The two
factors describe rotations around coupled parallel axes of equal speed but in opposite
direction so that the resulting motion is translational, as seen in Figure \ref{fig:parallelogram}. In the following we
will show that this is indeed the only one.

Let us first define the notion of trivial factorizations.
For this we want to call a motion trivial if it is a possibly non-linear reparametrization of a simple motion.

\begin{definition}
  \label{def:trivial}
  The motion \(C=(t-h_1)(t-h_2)\) is called \emph{trivial} if there is a
  map \(f\colon \mathbb{R} \to \mathbb{R}\) and \(h \in \ensuremath{\mathrm{CGA}_+}\) such that \(h\tilde{h}\in \mathbb{R}\) and \(C(t) = f(t)-h\).
\end{definition}

A typical example of a trivial motion is \(C = (t-h)(at + b - h)\) with
\(C\tilde{C} \in \mathbb{R}[t]\setminus\{0\}\) and some \(a, b \in \mathbb{R}\).

We proceed by distinguishing two cases, namely \(h_2\) being a rotation, or a
translation. For each of these two cases we can then calculate the space of
possible \(h_1\) for the conditions \eqref{eq:2}, \eqref{eq:3} from
Subsection~\ref{sec:irregular-factorizations}. Doing this reveals that the
solution space is very small and it only contains the circular translation and
trivial motions. These are applying the same rotation twice,
which just corresponds to a non-linear reparametrization of the basic rotation
and two translations which combined give a translation into a different
direction.

\begin{theorem}[Circular Translation]
  \label{thm:CircTrans}
  Up to Euclidean equivalence the circular translation \(C \coloneqq (t-((x
  e_{1}+y e_{2})(e_-+e_+)-e_{12}))(t-e_{12})\) is the only non-trivial quadratic
  Euclidean motion that has irregular factorizations.
\end{theorem}

\begin{proof}
  To study Euclidean motions as special conformal motions, we reduce \ensuremath{\mathrm{CGA}} to the
  algebra \(\mathbb{DH}\) of dual quaternions, given as
  \begin{multline*}
    x \coloneqq
    q_s
    -q_i e_{23}
    +q_j e_{13}
    -q_k e_{12}\\
    +p_s(e_{123+}+e_{123-})
    +p_i (e_{1+}+e_{1-})
    +p_j(e_{2+}+e_{2-})
    +p_k(e_{3+}+e_{3-}).
  \end{multline*}

  For \((t-x) \in \mathbb{DH}[t]\) to be a motion polynomial we need to enforce that
  \(x\) lies on the Study variety and \(x+\tilde x\) is real. Assuming \(p_i\neq
  0\), this gives us the solution \(x=\frac{q_j p_j + q_k p_k}{p_i} e_{23} + q_j
  e_{13} - q_k e_{12} + p_i e_{1m} + p_i e_{1p}+ p_k e_{3m} + p_k e_{3p} + p_s
  e_{123m}\). In the next step we need to specify our right factor and solve for
  the irregular factorizability condition \eqref{eq:2}.

  There are only two cases, rotation and translation. We will first look at the
  case, where the right factor is a rotation. Without loss of generality we take
  \(h_2=e_{12}\). For infinite factorizability we need
  \[\det((x-\tilde h_2)(\tilde x- h_2)) =\biggl(q_s^2 + \Bigl(\frac{q_j p_j+q_k p_k}{p_i}\Bigr)^2 + q_j^2 + (1-q_k)^2\biggr)^4=0.\]
  (Recall that for calculation we use the determinant of the quadrance,
  cf.~Remark~\ref{rem:two-things}.) Solving this gives \(x=-e_{12} + p_i
  (e_{1-}+e_{1+}) + p_j (e_{2-}+e_{2+})\). We can see that this solution is of
  the desired form. Now we can repeat this procedure assuming \(p_i=0\). In this
  case we get \(x=-e_{12} + p_j (e_{2-}+e_{2+})\), assuming \(p_j\neq 0\).
  Assuming also \(p_j = 0\) we arrive at the last rotational case giving us no
  real solution.

  This shows that there exists only the circular translation assuming that the
  second factor is a rotation.

  We now need to take care of the case that the second factor is a translation.
  Using the same setup as before with the difference that \(h_2 =
  e_{3+}+e_{3-}\), we once again get three cases.
  \begin{itemize}
  \item Case 1: \(p_i \neq 0\). \(x=p_i (e_{1+}+e_{1-}) + p_j (e_{2+}+e_{2-}) +
    p_k (e_{3+}+e_{3-})\). This corresponds to a second
    translation, in total giving a new translation along a different direction.
    Hence a trivial motion.
  \item Case 2: \(p_i=0, p_j \neq 0\) \(x= p_j (e_{2+}+e_{2-}) + p_k
    (e_{3+}+e_{3-})\). This solution is subsumed by the first case.
  \item Case 3: \(p_i=0, p_j = 0\). \(x=0\). This corresponds to a trivial motion.
  \end{itemize}

  In conclusion, there exist only trivially irregularly factorizable motions with
  translations as factors and only the circular translation when there is a
  rotation as a factor.
\end{proof}

\section{Polynomials With Commuting Factors}
\label{sec:commuting}

After having investigated the special case of rigid body motions, we will now
turn to the case of \emph{commuting} factors. There is also only one
known example which has been previously described, in \cite{dorst19,icgg_paper}.
In this case we get the extra condition that \(h_1h_2 = h_2h_1\). After
splitting the set of potential right factors \((t-h_2)\) into the parts where
\(h_2\) is a rotation, translation or scaling, we look for solutions. Doing this
we find the following holds true.

\begin{theorem}[Villarceau Motion]
  Let \(C \in \ensuremath{\mathrm{CGA}_+}[t]\) be an irregularly factorizable motion polynomial of
  degree two with commuting factors. Then \(C\) is either a trivial motion or
  conformally equivalent to the Villarceau motion \(C\coloneqq
  (t-e_{12})(t-e_{3+})\).
\end{theorem}

\begin{proof}
  We proceed in the same manner as in the proof of Theorem~\ref{thm:CircTrans}.
  Let \(x = x_1 + x_2 e_{12} + x_3 e_{13} + \cdots + x_{16} e_{23+-}\). This
  time we have the extra condition that \(x h_2-h_2 x=0\). Let us first assume
  \(h_2\) to describe a rotation. Without loss of generality \(h_2 = e_{12}\).
  We now get two possible solutions for \(x\) such that \((t-x)\) is a motion
  polynomial that commutes with \((t-h_2)\).
  \[
    x = x_{1} + x_{9} e_{3+} + x_{10} e_{3-} + x_{11} e_{+-}, \quad
    x^\prime = x_{1} + x_{2} e_{12}.
  \]
  We see that \(x^\prime\) just corresponds to the same rotation in a possibly
  different parametrization. This gives rise to a trivial motion and can
  therefore be disregarded.

  In the other case the irregular factorizability condition for \(x\) boils down to 
  \[((x_{1}^2 + x_{9}^2 - x_{10}^2 - x_{11}^2 + 1)^2-4(x_{9}^2 - x_{10}^2 -
    x_{11}^2))^2 = 0.\] 
    The left side equals to a square of a sum of squares as $4 x_{1}^2 +(x_{1}^2+x_{9}^2-x_{10}^2-x_{11}^2-1)^2$ which shows that in order for the solution to be
  real, we need \(x_{1} = 0\) and \(x_{9}^2 - x_{10}^2 - x_{11}^2=1\), which yields \(x_{9}=
  \pm \sqrt{x_{10}^2 + x_{11}^2 + 1}\). Then \[x =\pm\sqrt{x_{10}^2 + x_{11}^2 +
      1}\; e_{3+} + x_{10} e_{3-} + x_{11} e_{+-}\] is a non-trivial real
  solution. We will be doing the calculations for
  \begin{equation*}
  x=\sqrt{x_{10}^2 + x_{11}^2
    + 1}\; e_{3+} + x_{10} e_{3-} + x_{11} e_{+-} 
    \end{equation*}
    as the other case can be
  calculated analogously.

  To understand what motion \(x\) describes, we decompose it as a blade. This
  results in \(x =\alpha (b_1 \wedge b_2)\) where
  \begin{align*}
    \alpha &= \frac{1}{(x_{11}^2 + 1)},\\
    b_1 &= -x_{11}\Big(\sqrt{x_{10}^2 + x_{11}^2 + 1}\Big) \; e_- + (x_{11}^2+1) e_3 - x_{10}x_{11} e_+,\\
    b_2 &= \sqrt{x_{10}^2 + x_{11}^2 + 1} \; e_+ + x_{10} e_-.
  \end{align*}
  By forming a linear combination of \(b_1\) and \(b_2\) we get \(x = \alpha (p
  \wedge b_2)\) with
  \begin{equation*}
    p = b_1 - x_{11}b_2
      = (x_{11}^2+1)e_3-\Big(\sqrt{x_{10}^2+x_{11}^2+1}+x_{10}\Big)x_{11}e_+.
  \end{equation*}
  We can check that \(b_2\) is a sphere centered at the origin with non-zero
  radius and \(p\) is a plane with normal vector \(n\) and distance to origin
  \(d\), where
  \begin{gather*}
    n = (0,0,1),\quad
    d = -\bigg(\frac{(\sqrt{x_{10}^2+x_{11}^2+1}+x_{10})x_{11}}{x_{11}^2+1}\bigg).
  \end{gather*}
We now define a translation
  \[a \coloneqq 1-\bigg(\frac{\big(\sqrt{x_{10}^2+x_{11}^2+1}+x_{10}\big)x_{11}}{2(x_{11}^2+1)}\bigg)(e_{3+}+e_{3-})\]
  in the \(n\) direction by \(-d\) and observe the following properties:
  \begin{gather*}
    a\tilde{a} = a\wedge\tilde{a}=1,\quad
    a h_2\tilde{a} = h_2,\quad
    a p \tilde{a} = \frac{e_3}{\alpha}.
  \end{gather*}
  With this we can see that
  \[ax\tilde{a} = \alpha a(p\wedge b_2)\tilde{a}=\alpha a(p\wedge b_2)\tilde{a}=\alpha (ap\tilde{a})\wedge (ab_2\tilde{a})= e_3\wedge (ab_2\tilde{a})\]
  and
  \[ a(t-x)(t-h_2)\tilde{a} = (t-e_3\wedge (ab_2\tilde{a}))(t-h_2).\] Next we can
  investigate \(ab_2\tilde{a}\). Calculation shows that this corresponds to a
  sphere with center on the line \(e_{12}\) and some radius. Let us define
  \[b_2^\prime \coloneqq \frac{x_{11}}{\sqrt{x_{10}^2+x_{11}^2+1}-x_{10}} e_3.\]
  We check that \((t-e_3\wedge (ab_2\tilde{a}))=(t-e_3\wedge b_2^\prime)).\) We
  then scale \(x\) appropriately with scaling \(s\) centered at the origin. Such
  a scaling preserves planes and lines through the origin and therefore only
  changes the radius of \(ab_2\tilde{a}\), when applied to \((t-e_3\wedge
  (ab_2\tilde{a}))(t-h_2)\). We choose the scaling factor such that \(ab_2\tilde{a}\)
  is transformed to the unit sphere. In total this now gives us
  \[sa(t-x)(t-h_2)\tilde{s}\tilde{a}=(t-e_3\wedge
    e_+)(t-e_{12})=(t-e_{3+})(t-e_{12}).\] As both factors commute, we see that
  every irregularly factorizable motion with commuting factors, one of which is
  a rotation, is immediately the Villarceau motion.

  Let us now investigate the case of a transversion being the right factor.
  Without loss of generality we choose \(h_2 \coloneqq e_{3+}+e_{3-}\).
Similar to the rotational case we can find two possible commuting factors
  \((t-x)\) and \((t-x^\prime)\) of a motion polynomial. The infinite
  factorizability conditions read
  \[x_1^8 = 0, \quad ({x^{\prime}_1}^2+{x^{\prime}_2}^2)^4 = 0\] for \(x=x_1 +
    x_5(e_{1+}+e_{1-})+ x_8(e_{2+}+e_{2-})+ x_{10}(e_{3+}+e_{3-})\) and
    \(x^\prime=x^{\prime}_1 + x^{\prime}_2 e_{12} + x^{\prime}_5(e_{1+} + e_{1-}) +
    x^{\prime}_8(e_{2+} + e_{2-})\), respectively. We can immediately see that all
    real solutions give rise to trivial motions as \(x\) and \(x'\) both are translations.

  In the case of a scaling as a right factor we get the following: Without loss
  of generality let \(h_2 \coloneqq -e_{+-}\). Then analogously to the previous
  cases the irregular factorizability conditions for \(x=x_1 + x_2 e_{12} + x_3
  e_{13} + x_6 e_{23}\) and \(x^\prime = x_1 + x_{11} e_{+-}\) read
  \[(x_1^2+x_2^2+x_3^2+x_6^2-1)^2+4(x_2^2+x_3^2+x_6^2) = 0,\quad
    ({x^{\prime}_1}^2-(x^{\prime}_{11}-1)^2)^4 = 0.\] In the first case we get
  \(x=1\). This generates a trivially factorizable motion. For the second case we get
  \(x^\prime = x^{\prime}_1 + (1 \pm x^{\prime}_1)e_{+-}\), which is just an offset
  of the original scaling. Hence, also a trivial factorization.
\end{proof}

\section{Conclusion}

In this text we have completely characterized quadratic motion polynomials with irregular factorizations. The restriction to polynomials of degree two allows for relatively simple computational approaches and includes already known quadratic motion polynomials. Extensions to higher degrees are of course of interest as are extensions to more general motion groups. We expect many similarities but also foresee some crucial differences, for example the algebraic description of the Study variety. As we have seen, irregularly factorizable polynomials can have infinitely many but also finitely many factorizations. We consider a characterization of the infinitely factorizable motion polynomials as an worthy topic of further research. For factorization of motion polynomials, we are interested in the decomposition of rational curves on those orthogonal groups  comparing both with the 2-blades decomposition \cite{dorst16} of elements of those orthogonal groups. Finally, as suggested by a reviewer, it 
would be interesting to draw parallels to the work of Roelfs and De Keninck \cite{Roelfs23} on bivector decomposition in future work.

\section*{Data Availability Statement}

No new data was created or analyzed in this study.

\section*{Acknowledgments}

Zijia Li is supported by the National Key R\&D Program of China (2023YFA10
09401) and partially supported by the Strategic Priority Research Program of the
Chinese Academy of Sciences 0640000 \& XDB0640200 and the Guangdong Basic and
Applied Basic Research Foundation under Grant 2024A151501 0506.

This research was funded in whole or in part by the Austrian Science Fund (FWF)  P~33397-N/Grant DOI: 10.55776/P33397 (Rotor Polynomials: Algebra and Geometry of Conformal Motions).

This research was funded in whole or in part by the Austrian Science Fund (FWF) 10.55776/I6233. For open access purposes, the author has applied a CC BY public copyright license to any author-accepted manuscript version arising from this submission.
 
\bibliographystyle{elsarticle-num}

\end{document}